\documentclass[11pt]{article} 
\usepackage[utf8]{inputenc}
\usepackage{caption,graphicx,amsmath,amsthm,amssymb, authblk,xcolor,enumitem,subfigure,float,verbatim,hyperref,pgfplots,bibentry,mathtools,soul,isomath, cancel}
\usepackage{mathrsfs}
\usepackage{indentfirst}
\usepackage{nccmath}
\graphicspath{ {images/} }
\usepackage{multicol, MnSymbol}
\usepackage{graphicx,amsmath,amssymb,amsthm,marvosym,tikz,fontenc}
\usetikzlibrary{backgrounds}
\usetikzlibrary{patterns, shapes.misc, positioning}
\usepackage{pgfplots}
\pgfplotsset{width=10cm,compat=1.9,tick scale binop=\times}
\usepackage{indentfirst}
\usepackage[a4paper,bindingoffset=0.2in,%
left=0.8in,right=1in,top=1.2in,bottom=1.2in,%
footskip=.25in]{geometry}
\usepackage{relsize}
\usepackage[english]{babel}
\allowdisplaybreaks
\theoremstyle{plain}
\newtheorem{theorem}{Theorem}[section]
\newtheorem{lemma}[theorem]{Lemma}

\newtheorem{corollary}[theorem]{Corollary}
\newtheorem{definition}[theorem]{Definition}
\newtheorem{remark}[theorem]{Remark}
\newtheorem{note}[theorem]{Note}

\date{}
\begin{document}
\title
{\bf{Eccentricity energy change of coalescence of graphs due to edge deletion}}
\author {\small Anjitha Ashokan
\footnote{anjithaashokan1996@gmail.com}   and Chithra A V
\footnote{chithra@nitc.ac.in} \\ \small Department of Mathematics, National Institute of Technology Calicut,\\
\small Kerala, India-673601}
\date{}
\maketitle
\begin{abstract} 
The eccentricity matrix of a graph is obtained from the distance matrix by keeping the largest entries in their row or column, and the remaining entries are replaced by zeros. The eccentricity energy of a graph is the sum of the absolute values of the eigenvalues of its eccentricity matrix. 
In this paper, we investigate the effect of edge deletion on the eccentricity energy of  graphs of the form
$$G=K_{2n}\circ_{n} K_{2n}\circ_{n}\cdots \circ_{n} K_{2n}, (\text{\textit{l} copies of } K_{2n}),$$ 
where $n\geq 3,$ $l\geq 2,$ and $\circ_{n}$ denotes the $n-$coalescence of graphs, and prove that the eccentricity energy increases whenever an edge is removed. This result identifies a class of graphs whose eccentricity energy exhibits monotonic growth under edge deletion.
\end{abstract}

\section{Introduction}

Let $G$ be a finite, simple, and connected graph on $n$ vertices with vertex set $V=\{v_{1},v_{2},\ldots,v_{n}\}$ and $E\subseteq V\times V$ as the  edge set. The complete graph on $n$ vertices is denoted by $K_{n}.$ 
$A(G)$ denotes the $n\times n$ \textit{adjacency matrix} of $G,$
whose rows and columns are indexed by the vertex set of $G$ and the entries are defined as, $A(G)=(a_{ij}),$ where $a_{ij}=1$ if the vertices $v_{i},$ $v_{j}$ are adjacent and $0$ otherwise.
The \textit{energy} of $G$ is defined as $E_{A}(G)=\sum_{i=1}^{n}|\lambda_{i}|,$ where $\lambda_{1},\lambda_{2},\ldots,\lambda_{n}$ are the eigenvalues of $A(G).$
The \textit{distance} between two vertices $v_{i},v_{j},$ denoted by $d(v_{i},v_{j})$ is defined as the length of the shortest path between $v_{i}$ and $v_{j}.$ 
$D(G)$ denotes the   \textit{distance matrix}  of $G$ and is defined as  $D(G)=(d_{ij}),$ where $d_{ij}=d(v_{i},v_{j}).$
Let $\gamma_{1},\gamma_{2},\ldots,\gamma_{n}$ be the eigenvalues of $D(G),$ then the \textit{distance energy} of $G$ is  defined as, $E_{D}(G)=\sum_{i=1}^{n}|\gamma_{i}|.$\\
The \textit{eccentricity} of a vertex $v_{i},$ $e(v_{i}),$ is the largest distance from $v_{i}$ to all other vertices.
\textit{Eccentricity matrix}  of $G,$ $\epsilon(G)=(\epsilon_{ij})$ is an $n\times n$ symmetric matrix, where 
\begin{equation*}
   \epsilon_{ij}=\begin{cases}
            d(v_{i},v_{j}) & \text{ if } d(v_{i},v_{j})=\min\{e(v_{i}),e(v_{j})\},\\
            0 &\text{ otherwise.}
            \end{cases}
\end{equation*}
In 2013, Rand\'{c} introduced the eccentricity matrix as  $D_{MAX}$, and later it was renamed as eccentricity matrix by Wang et al.\cite{MR3906706} in $2018.$
  The set of all eigenvalues of $\epsilon(G)$ is the $\epsilon-$ spectrum of $G.$ 
 Let $\epsilon_{1},\epsilon_{2},\ldots,\epsilon_{n}$ be the eigenvalues of $\epsilon(G),$ known as $\epsilon-$ eigenvalues. The \textit{eccentricity energy} of $G$ is defined as $E_{\epsilon}(G)=\sum_{i=1}^{n}|\epsilon_{i}|.$
The \textit{$\epsilon-$spectral radius} of $G$ is the maximum of absolute values of $\epsilon-$eigenvalues of $G,$ and is denoted by $\rho_{\epsilon}(G).$
Unlike the adjacency and distance matrix, the eccentricity matrix is not always irreducible.\\

Let $e$ be an edge of $G$ such that $G\backslash\{e\}$ remains connected. 
The eccentricity energy of a graph may increase, decrease, or remain unchanged after the deletion of an edge. Consequently, the effect of edge deletion on graph energy is a non-trivial problem, and it might depend on the structure of the graph.
It is known that if a graph $G$ has a unique positive distance eigenvalue, then the distance energy of $G\backslash\{e\}$ 
is greater than the distance energy of $G$ \cite{zbMATH06124033}. However, this behavior need not hold in the case of eccentricity energy. For example, the graph $G$ shown in Figure \ref{fig:1} has exactly one positive
$\epsilon$-eigenvalue, but $E_{\epsilon}(G)=25.2664>20.159=E_{\epsilon}(G\backslash\{e\}).$ 
\begin{figure}[H]
 \centering
 \begin{tikzpicture}[scale=.5]

 \filldraw[fill=black](0,0)circle(0.1cm); 
 \filldraw[fill=black](-3,1.5)circle(0.1cm); 
 \filldraw[fill=black](-3,-1.5)circle(0.1cm); 
  \filldraw[fill=black](3,1.5)circle(0.1cm); 
   \filldraw[fill=black](3,-1.5)circle(0.1cm); 
    \filldraw[fill=black](-1,3)circle(0.1cm); 
     \filldraw[fill=black](1,3)circle(0.1cm); 
      \filldraw[fill=black](-1,-3)circle(0.1cm); 
       \filldraw[fill=black](1,-3)circle(0.1cm); 
  \draw(0, 0)--(-3,1.5);
  \draw(0, 0)--(-3,-1.5);
  \draw(0, 0)--(3,1.5);
  \draw(0, 0)--(3,-1.5);
  \draw(0, 0)--(-1,3);
  \draw(0, 0)--(1,3);
  \draw(0, 0)--(-1,-3);
  \draw(0, 0)--(1,-3);
   \draw(-3,1.5)--(-3,-1.5);
   \draw(-1,3)--(1,3);
   \draw(3,1.5)--(3,-1.5);
   \draw(-1,-3)--(1,-3);
\node at (2,1.45) {\bfseries e};   
   
 \end{tikzpicture}

  \caption{}
 \label{fig:1}
 \end{figure}
The variation of graph energy due to edge deletion for various graph matrices has been extensively studied in recent years. The change in adjacency energy under edge deletion was investigated in \cite{zbMATH05254234,shan2020energy,wang2015graph}. In \cite{zbMATH06891800}, the authors studied complete bipartite graphs and showed that their distance energy increases when an edge is removed. Moreover, they conjectured that complete 
$k$-partite graphs possess the same property, which was later proved in \cite{zbMATH07312043,zbMATH07125983}.
 The eccentricity energy change due to edge deletion in complete $k$-partite graphs was examined in \cite{zbMATH07476489}.\\

 Classifying graphs according to their eccentricity energy behavior under edge deletion is an interesting area of research.
 Motivated by these, we investigate the change in eccentricity energy of the coalescence of complete graphs. In particular, we prove that for
 $$G=K_{2n}\circ_{n} K_{2n}\circ_{n}\cdots \circ_{n} K_{2n},$$
where $G$ consists of $l$ copies of $K_{2n}$ with   $n\geq 3,$ $l\geq 2,$ the eccentricity energy of $G$ always increases under the deletion of an edge.

 Following are the prerequisites for the rest of the paper.
\begin{definition}$\textnormal{\cite{sudhir_r_jog_2021_4725009}}$
 Let $G_{1}$ and $G_{2}$ be connected graphs of orders $n_{1}$ and $n_{2}$, and sizes $m_{1}$ and $m_{2}$, respectively, each containing an induced complete subgraph $K_{k}$ with $k < \min\{n_{1}, n_{2}\}$. The $k$-coalescence of $G_{1}$ and $G_{2}$, denoted by $G_{1} \circ_{k} G_{2}$, is obtained by identifying the corresponding vertices of these induced $K_{k}$ subgraphs. The order and size of 
$G_{1} \circ_{k} G_{2}$ are $n_{1}+n_{2}-k$ and $m_{1}+m_{2}-\binom{k}{2}$, respectively.
\end{definition}
As a generalization, the coalescence operation can be extended to $l\geq 2$ pairwise disjoint graphs of orders $n_{1},n_{2},\ldots, n_{l}$ by identifying an induced complete subgraph  $K_{k},$ $ k < \min\{n_{1}, n_{2},\ldots,n_{l}\}$ from each graph. The resulting graph is denoted by $G_{1}\circ_{k} G_{2}\circ_{k} \cdots \circ_{k} G_{l}.$

\begin{theorem}(Perron-Frobenius theorem)\label{Perron-Frobenius Theorem}\textnormal{\cite{zbMATH05908284}}
     Let $T$ be an irreducible matrix. If $0\leq S\leq T, S\neq T$, then every eigenvalue $\gamma$ of $S$ satisfies, $|\gamma|\leq \lambda_{max},$  where $\lambda_{max}$ is the largest eigenvalue of $T.$
\end{theorem}
\begin{theorem}(Schur Complement formula)\label{Schurs complement formula}$\textnormal{\cite{cvetkovic1980spectra}}$
Let $U, V, W$ and $X$ be matrices with $U$ invertible. Let $$S=\begin{pmatrix}
    U&V\\
    W&X
\end{pmatrix}.$$
Then, $det(S)=det(U)det(X-WU^{-1}V).$
\end{theorem}
\begin{definition}(Equitable partition)$\textnormal{\cite{brouwer2011spectra}}$
 Suppose that $M$ is a square matrix whose rows and columns are indexed by $X=\{1,2,\ldots, n\}.$ Let $\Pi$ be a partition of $X,$  $X=X_{1}\cup X_{2} \cup \cdots \cup X_{m}$. The matrix $M$ can be written as 
\begin{equation*}
M=\begin{pmatrix}
    M_{1,1} & M_{1,2}&\cdots &M_{1,m}\\
    M_{2,1} &M_{2,2} &\cdots &M_{2,m}\\
    \vdots&\vdots&\ddots&\vdots\\
    M_{m,1}&M_{m,2}&\cdots&M_{m,m}
\end{pmatrix}
\end{equation*}
where $M_{i,j}$ is the submatrix of $M$ whose rows and columns are indexed by $X_{i}$ and $X_{j},$ respectively, for $1\leq i,j\leq m.$ Let $b_{ij}$ be the average row sum of $M_{i,j}.$ Then, $Q_{m}=(b_{ij})_{m\times m}$ is the quotient matrix of $M$ with respect to the partition $\Pi.$ If the row sum of each block $M_{i,j}$ is a constant, then the partition is an equitable partition. 
\end{definition}

\begin{theorem}\label{equitable quotient matrix eigenvalue thm}$\textnormal{\cite{brouwer2011spectra}}$
Let $M$ be a square matrix and $\Pi$ be an equitable partition of $M$ with quotient matrix $B$. Then, the multiset of eigenvalues of $B$ is contained in the multiset of eigenvalues of $M.$
\end{theorem}
\begin{theorem}\label{spectral radius of a matrix and its equitatable quotient matrix are the same}$\textnormal{\cite{zbMATH07080161}}$
 Let $B$ be the equitable quotient matrix of $M.$ If  $B$ is irreducible and nonnegative matrix, then $M$ and $B$ have the same spectral radius.

\end{theorem}
The \textit{eccentric graph}, $G^{e}$ of a connected graph $G$ is defined as the graph having the same vertex set as $G,$ that is, $V(G^{e})=V(G),$ where two vertices are adjacent if and only if the distance between them equals the minimum of their eccentricities.
\begin{lemma}$\textnormal{\cite{MR4522999}}$\label{eccentric graph}
   The graph $G^{e}$ is connected if and only if the matrix $\epsilon(G)$ is an irreducible matrix.
\end{lemma}
\section{Eccentricity spectrum of coalescence of graphs}
In this section, we investigate the eccentricity spectrum of the coalescence of complete graphs. We compute the eccentricity spectrum, inertia, and eccentricity energy of graphs of the form $$G=K_{2n}\circ_{n}K_{2n}\circ_{n}\cdots \circ_{n}K_{2n}, ( l \text{ copies of } K_{2n}), 
l\geq 2, n\geq 2.$$
We observe that the graph  $G=K_{a_{1}}\circ_{k}K_{a_{2}}\circ_{k}\cdots \circ_{k}K_{a_{l}}$, with $a_{i}\geq 3,$ for $i=1,2,\ldots,l $ has diameter $2.$

\begin{lemma}\label{irreducibility of general coalesence of graphs}
Let $G=K_{a_{1}}\circ_{k}K_{a_{2}}\circ_{k}\cdots \circ_{k}K_{a_{l}}$, $l\geq 2,$ $a_{i}\geq 3,$ for $i=1,2,\ldots,l$, then
$\epsilon(K_{a_{1}}\circ_{k}K_{a_{2}}\circ_{k}\cdots \circ_{k}K_{a_{l}})$ is irreducible.
\end{lemma}
\begin{proof}
Let $$V(G)=V(K_{k})\cup V(K_{a_{1}}\backslash K_{k})\cup V(K_{a_{2}}\backslash K_{k})\cdots \cup V(K_{a_{l}}\backslash K_{k})=V(G^{e}),$$ where $K_{k}$ is the complete subgraph identified in each of the complete graphs $K_{a_{i}},1\leq i \leq l.$\\
For any vertex $v \in V(K_{k}),$ we have $e(v)=1.$ That is, for any $v_{1},v_{2}\in V(K_{k}),$ $$d(v_{1},v_{2})=1=min\{e(v_{1}),e(v_{2})\}.$$ Hence the vertex set $V(K_{k})$ induces a complete subgraph in $G^{e}$. \\
Let $u\in V(G)\backslash V(K_{k}).$
 Then,  for every $u\in V(K_{k}),$    $$d(u,v)=1=min\{e(u),e(v)\}.$$
 Therefore, each vertex outside $V(K_{k})$ is adjacent to every vertex of $V(K_{k})$ in $G^{e}.$
 Consequently, the eccentric graph $G^{e}$ is connected.
 By Lemma \ref{eccentric graph}, it follows that the eccentricity matrix $\epsilon(G)$ is irreducible.

\end{proof}
\begin{theorem}\label{k coalesence,general,l times }
   Let $G=K_{a_{1}}\circ_{k}K_{a_{2}}\circ_{k}\cdots \circ_{k}K_{a_{l}}$, $l\geq 2$, $a_{i}\geq 3,$ for $i=1,2,\ldots,l. $ Then, $\epsilon-$spectrum of $G$ consists of: \begin{enumerate}
       \item $-1$ with multiplicity at least $k-1.$
       \item $0$ with multiplicity at least $(a_{1}-k-1)+(a_{2}-k-1)+\cdots+(a_{l}-k-1).$
       \item the eigenvalues of the matrix $\begin{pmatrix}
           k-1& a_{1}-k & a_{2}-k &\cdots &a_{l}-k \\
           k  & 0  & 2(a_{2}-k)   &\cdots &2(a_{l}-k)\\
           k  & 2(a_{1}-k)    &0  &\cdots&2(a_{l}-k)\\
           \vdots&\vdots&\vdots   &\ddots&\vdots\\
           k&2(a_{1}-k)&2(a_{2}-k)&\cdots&0
       \end{pmatrix}.$
   \end{enumerate} 
\end{theorem}
\begin{proof}
    By suitable labeling of vertices, we get 
    $$\epsilon(G)=\begin{pmatrix}
        (J-I)_{k\times k}&J_{k \times a_{1}-k}&J_{k\times a_{2}-k}&\cdots &J_{k\times a_{l}-k}\\
        J_{a_{1}-k\times k} & 0_{a_{1}-k\times a_{1}-k}& 2J_{a_{1}-k\times a_{2}-k}&\cdots &2J_{a_{1}-k\times a_{l}-k}\\
         J_{a_{2}-k \times k }& 2J_{a_{2}-k\times a_{1}-k}&  0_{a_{2}-k\times a_{2}-k}& \cdots &2J_{a_{2}-k\times a_{l}-k}\\
         \vdots &  \vdots &  \vdots &\ddots&\vdots\\
         J_{a_{l}-k \times k} & 2J_{a_{l}-k\times a_{1}-k}&2J_{a_{l}-k\times a_{2}-k}& \cdots &0_{a_{l}-k\times a_{l}-k}
         \end{pmatrix}_{(k+\sum_{i=1}^{l}(a_{i}-k))\times (k+\sum_{i=1}^{l}(a_{i}-k)) }$$
For $i=2,3,\ldots ,k,$ consider the vector 
 $$x^{(i)}=\big(  x_{1,1},x_{1,2},\ldots,x_{1,k},x_{2,1},x_{2,2},\ldots,x_{2,a_{1}-k}, x_{3,1},x_{3,2},\ldots,x_{3,a_{2}-k},\cdots, x_{l+1,1},x_{l+1,2},\ldots,x_{l+1,a_{l}-k}      \big)$$
such that $x_{1,1}=1,x_{1,i}=-1$ and $x_{u,v}=0,$ for all $x_{u,v}\neq  x_{1,1},x_{1,i}.$
Then, $\epsilon(G)x^{(i)}=(-1)x^{(i)}.$ 
Thus, $-1$ is an eigenvalue of $\epsilon(G)$ with multiplicity at least $k-1.$\\
\\
For $j=2,3,\ldots ,l+1,$ consider the vector 
$$y^{(j,t)}=\big( y_{1,1},y_{1,2},\ldots,y_{1,k},y_{2,1},y_{2,2},\ldots,y_{2,a_{1}-k}, y_{3,1},y_{3,2},\ldots,y_{3,a_{2}-k},\cdots, y_{l+1,1},y_{l+1,2},\ldots,y_{l+1,a_{l}-k}      \big)$$
such that $y_{j,1}=1,y_{j,t}=-1,$ for $t=2,3,\ldots ,a_{j-1}-k,$ and $y_{u,v}=0,$ for all $y_{u,v}\neq y_{j,1},y_{j,t}.$
Then, $\epsilon(G)y^{(j,t)}=0.$ 
Thus, $0$ is an eigenvalue of $\epsilon(G)$ with multiplicity at least $(a_{1}-k-1)+(a_{2}-k-1)+\cdots+(a_{l}-k-1).$\\
Note that all the eigenvectors constructed so far are orthogonal to\\
$$\begin{pmatrix}
    J_{k\times 1}\\
    0_{a_{1}-k\times 1}\\
    0_{a_{2}-k\times 1}\\
    \vdots\\
    0_{a_{l}-k\times 1}
\end{pmatrix},
\begin{pmatrix}
    0_{k\times 1}\\
    J_{a_{1}-k\times 1}\\
    0_{a_{2}-k\times 1}\\
    \vdots\\
    0_{a_{l}-k\times 1}
\end{pmatrix},
\begin{pmatrix}
    0_{k\times 1}\\
    0_{a_{1}-k\times 1}\\
    J_{a_{2}-k\times 1}\\
    \vdots\\
    0_{a_{l}-k\times 1}
\end{pmatrix},
\ldots,\begin{pmatrix}
    0_{k\times 1}\\
    0_{a_{1}-k\times 1}\\
    0_{a_{2}-k\times 1}\\
    \vdots\\
    J_{a_{l}-k\times 1}
\end{pmatrix}.
$$

Since  $\epsilon(G)$ is a  symmetric matrix of order $k+\sum_{i=1}^{l}(a_{i}-k)$, $\mathbb{R}^{(k+\sum_{i=1}^{l}(a_{i}-k))}$ has an orthogonal basis consisting of eigenvectors of $\epsilon(G).$
Hence, the remaining $l+1$ eigenvectors are spanned by these $l+1$ vectors.
Thus, if $\sigma$ is an eigenvalue of $\epsilon(G)$ with an eigenvector $\zeta$ of the form $\begin{pmatrix}
    s_{1} J_{k\times 1}\\
    s_{2}J_{a_{1}-k\times 1}\\
    s_{3}J_{a_{2}-k\times 1}\\
    \vdots\\
    s_{l+1}J_{a_{l}-k\times 1}
    \end{pmatrix},$
 where $s_{i}\in \mathbb{R},$   
then from $\epsilon(G)\zeta=\sigma \zeta$ it follows  that the remaining $l+1$ eigenvalues are the eigenvalues of the matrix
$B_{1}=\begin{pmatrix}
           k-1& a_{1}-k & a_{2}-k &\cdots &a_{l}-k \\
           k  & 0  & 2(a_{2}-k)   &\cdots &2(a_{l}-k)\\
           k  & 2(a_{1}-k)    &0  &\cdots&2(a_{l}-k)\\
           \vdots&\vdots&\vdots   &\ddots&\vdots\\
           k&2(a_{1}-k)&2(a_{2}-k)&\cdots&0
       \end{pmatrix}.$


\end{proof}

\begin{corollary}
Let $G=K_{a_{1}}\circ_{k} K_{a_{2}},$ where $a_{1},a_{2}\geq 3.$ Then,  inertia of $ \epsilon(G),$ $$(n_{-}(\epsilon(G), n_{0}(\epsilon(G)),n_{+}(\epsilon(G))=(k+1,a_{1}+a_{2}-2k-2,1).$$    
\end{corollary}
\begin{proof}
Consider the matrix $B_{2}=\begin{pmatrix}
                  k-1&a_{1}-k&a_{2}-k\\
                  k  &0&2(a_{2}-k)\\
                  k &2(a_{1}-k)&0
                    \end{pmatrix}.$    
The characteristic polynomial of  $B_{2}$ is
$P_{B_{2}}(x)=x^{3}-(k-1)x^{2}-\big(  4(a_{1}-k)(a_{2}-k)+k(a_{2}-k)+k(a_{1}-k) \big)x-\big(4(a_{1}-k)(a_{2}-k)\big).$
Let $\alpha_{1},\alpha_{2},\alpha_{3}$ be the roots of $P_{B_{2}}(x).$ Then, we have the following relation between the roots and the  coefficients of $P_{B_{2}}(x):$\\
\begin{align*}
    \alpha_{1}+\alpha_{2}+\alpha_{3}=&k-1,\\
    \alpha_{1}\alpha_{2}+\alpha_{1}\alpha_{3}+\alpha_{2}\alpha_{3}=& -\big(  4(a_{1}-k)(a_{2}-k)+k(a_{2}-k)+k(a_{1}-k) \big),\\
    \alpha_{1}\alpha_{2}\alpha_{3}=&4(a_{1}-k)(a_{2}-k)>0.
\end{align*}
Since  $ \alpha_{1}\alpha_{2}\alpha_{3}>0,$ $P_{B_{2}}(x)$ has zero or two negative roots. But, $\alpha_{1}\alpha_{2}+\alpha_{1}\alpha_{3}+\alpha_{2}\alpha_{3}<0.$ Thus, $P_{B_{2}}(x)$ has exactly $2$ negative roots and $1$ positive root. 
Now, the result follows from Theorem \ref{k coalesence,general,l times }.
\end{proof}
\begin{corollary}\label{k2n l times coalesence spectra } 
Let $G=K_{2n}\circ_{n}K_{2n}\circ_{n}\cdots \circ_{n}K_{2n},$ ( $l$ copies of $K_{2n}$), $l\geq 2, n\geq 2.$
Then, the spectrum of $\epsilon(G)$ consists of $-1$ with multiplicity at least $n-1,$ $0$ with multiplicity at least $(n-1)l,$ and roots of the polynomial$ (x+2n)^{l-1}  \big(x^{2}+x(1-n(2l-1))+n^{2}(l-2)-2n(l-1)\big).$
\end{corollary}

\begin{proof}
Consider the matrix 
$B_{3}=\left(
\begin{array}{c|cccc}
n-1 & n & n & \cdots & n \\ \hline
n   & 0 & 2n & \cdots & 2n \\
n   & 2n & 0 & \cdots & 2n \\
\vdots & \vdots & \vdots & \ddots & \vdots \\
n   & 2n & 2n & \cdots & 0
\end{array}
\right)$.\\
By Lemma \ref{Schurs complement formula} we have,
\begin{align*}
det(B_{3}-xI)=& det(n-1-x)det\begin{pmatrix}
    -x-\frac{n^{2}}{n-1-x}&2n-\frac{n^{2}}{n-1-x}&\cdots&2n-\frac{n^{2}}{n-1-x}\\
    2n-\frac{n^{2}}{n-1-x}&    -x-\frac{n^{2}}{n-1-x}&\cdots& 2n-\frac{n^{2}}{n-1-x}\\
    \vdots & \vdots  & \ddots & \vdots \\
    2n-\frac{n^{2}}{n-1-x}&2n-\frac{n^{2}}{n-1-x}&\cdots&-x-\frac{n^{2}}{n-1-x}
    \end{pmatrix}\\
    =&(n-1-x)det\Big(  (-x-2n)I+(2n-\frac{n^{2}}{n-1-x}J        \Big)\\
    =&(-x-2n)^{l-1}  \big(x^{2}+x(1-n(2l-1))+n^{2}(l-2)-2n(l-1)\big).
   \end{align*}
    Hence, the result follows from Theorem \ref{k coalesence,general,l times }.
\end{proof}
The following corollary follows from Corollary \ref{k2n l times coalesence spectra }  and Theorem \ref{spectral radius of a matrix and its equitatable quotient matrix are the same}.
\begin{corollary}
 Let $G=K_{2n}\circ_{n}K_{2n}\circ_{n}\cdots \circ_{n}K_{2n},$ ( $l$ copies of $K_{2n}$), $l\geq 2, n\geq 2.$ Then, the spectral radius of $\epsilon(G),$ $\rho_{\epsilon}(G)=\frac{\big(n(2l-1)-1\big)+\sqrt{\big(1-n(2l-1)\big)^{2}-4\big(n^{2}(l-2)-2n(l-1)\big)}}{2}.$   
\end{corollary}
To determine the inertia and the eccentricity energy of $G=K_{2n}\circ_{n}K_{2n}\circ_{n}\cdots \circ_{n}K_{2n},$ ( $l$ copies of $K_{2n}$), $l\geq 2, n\geq 2,$ 
it suffices to examine the sign of
$\frac{\big(n(2l-1)-1\big)-\sqrt{\big(1-n(2l-1)\big)^{2}-4\big(n^{2}(l-2)-2n(l-1)\big)}}{2}.$ This leads to the following corollaries.
\begin{corollary}\label{inertia for n coalecence of K2n, for all cases}
 Let $G=K_{2n}\circ_{n}K_{2n}\circ_{n}\cdots \circ_{n}K_{2n},$ ( $l$ copies of $K_{2n}$), $l\geq 2, n\geq 2.$ Then,   the inertia of $\epsilon(G),$  
$(n_{-}(\epsilon(G), n_{0}(\epsilon(G)),n_{+}(\epsilon(G))=\begin{cases}
            (n+1,2(n-1),1), & \text{ if }l=2\\
            (5,10,1), &\text{ if } l=3,n=4\\
            (n+1,3(n-1),2), &\text{ if } l=3,n>4\\
            (n+2,3(n-1),1), &\text{ if } l=3,n<4\\
            (l+1,l,1), &\text{ if } l\geq 4 ,n=2\\
            (n+l-2,(n-1)l,2), &\text{ if } l> 4,n\geq 3\\
            (5,9,1), &\text{ if } l= 4,n=3\\
            (5,4,1), &\text{ if } l= 4,n=2\\
            (n+l-2,(n-1)l,2), &\text{ if } l=4,n> 3.\\
            \end{cases}$
\end{corollary}
\begin{proof}
By Corollary \ref{k2n l times coalesence spectra }, the eccentricity matrix,
$$\epsilon(K_{2n}\circ_{n}K_{2n}\circ_{n}\cdots \circ_{n}K_{2n}),$$
has at least $n+l-2$ negative eigenvalues, $(n-1)l$ zero eigenvalues, and at least $1$ positive eigenvalue. 
Hence, to determine the inertia of $\epsilon(K_{2n}\circ_{n}K_{2n}\circ_{n}\cdots \circ_{n}K_{2n}),$ it is enough to find the sign of $x_{l,n}$ in each cases, where  $$x_{l,n}=\frac{\big(n(2l-1)-1\big)-\sqrt{\big(1-n(2l-1)\big)^{2}-4\big(n^{2}(l-2)-2n(l-1)\big)}}{2}.$$\\

For $l=2,$ $$x_{2,n}=\frac{(3n-1)-\sqrt{(1-3n)^{2}+8n}}{2}<0.$$\\

For $l=3,$ $$x_{3,n}=\frac{(5n-1)-\sqrt{(1-5n)^{2}-4n(n-4)}}{2}.$$
Clearly, for $n=4,$ $x_{3,4}=0.$ Moreover, if $n>4,$ then $x_{3,l}>0,$ and if $n<4,$ then $x_{3,n}<0.$ \\

For $l\geq 4$ and $n=2,$ $$x_{l,2}=\frac{(4n-3)-\sqrt{(3-4n)^{2}+16}}{2}<0$$\\

For $l>4$ and $n\geq 3,$   suppose  $$x_{l,n}=\frac{\big(n(2l-1)-1\big)-\sqrt{\big(1-n(2l-1)\big)^{2}-4\big(n^{2}(l-2)-2n(l-1)\big)}}{2}<0,$$ solving this inequality yields $(n-2)(l-2)<2,$ which is a contradiction. Hence, for $l>4$ and $n\geq3,$ $x_{l,n}>0.$\\

Finally, for $l=4,$ $$x_{l,n}=\frac{(7n-1)-\sqrt{(1-7n)^{2}-8n(n-3)}}{2}.$$ Clearly, for
$n=3,$ $x_{4,3}=0.$ Moreover,  $x_{4,2}<0,$ and for $n>3,$ $x_{4,n}>0.$ 
Thus, the result follows.
\end{proof}

\begin{corollary}\label{energy k coalencence for all cases}
  Let $G=K_{2n}\circ_{n}K_{2n}\circ_{n}\cdots \circ_{n}K_{2n},$ ( $l$ copies of $K_{2n}$), $l\geq 2, n\geq 2.$ Then,   the  $\epsilon$-energy of $G$ is $$E_{\epsilon}(G)=\begin{cases}
                    2\big(n(2l-1)-1\big), & \text{ if }l=3 \text{ and } n>4 \\
                                          & \text{ or }  l> 4 \text{ and } n\geq 3\\
                                          &\text{ or }  l= 4 \text{ and } n>3\\
                    \big(n(2l-1)-1\big)+\sqrt{\big(1-n(2l-1)\big)^{2}-4\big(n^{2}(l-2)-2n(l-1)\big)}, & \text{ otherwise. }                     
                      \end{cases}$$
\end{corollary}
\section{Eccentricity energy change of coalescence of complete graphs due to edge deletion}
In this section we estimate the eccentricity energy of $K_{2n}\circ_{n} K_{2n}\circ_{n} \ldots \circ_{n} K_{2n}\backslash\{e\}$.\\
Note that for any edge $e \in K_{2n}\circ_{n} K_{2n}\circ_{n}\cdots\circ_{n} K_{2n}$, exactly one of the following three cases occurs.

\medskip
\noindent\textbf{Case 1.} \textit{Edges induced by the common clique $K_{n}$.}  
These are the edges of the cliques which are identified corresponding to each of the graphs $K_{2n}.$
\medskip

\noindent\textbf{Case 2.} \textit{Edges incident with the common clique $K_{n}$.}  
Each edge in this case has one endpoint in $K_{n}$ and the other in $V(K_{2n}) \backslash V(K_{n})$. 
\medskip

\noindent\textbf{Case 3.} \textit{Edges not incident with the common clique $K_{n}$.}  
Both endpoints of such edges lie in $V(K_{2n}) \backslash V(K_{n})$.

\subsection{Eccentricity energy of  \texorpdfstring{$K_{2n}\circ_{n} K_{2n}\circ_{n}\ldots \circ_{n} K_{2n}\backslash\{e\}$}  ffor case 1 }
We first compute the eccentricity eigenvalues of $K_{2n}\circ_{n} K_{2n}\circ_{n}\ldots \circ_{n} K_{2n}\backslash\{e\}$. 
We then compare the eccentricity energy of $K_{2n}\circ_{n} K_{2n}\circ_{n}\ldots \circ_{n} K_{2n}$ with that of $K_{2n}\circ_{n} K_{2n}\ldots \circ_{n} K_{2n}\backslash\{e\}$.

\begin{lemma}\label{spectrum of coalesence, general, case 1}
 Let $G=K_{2n}\circ_{n} K_{2n}\circ_{n}\cdots \circ_{n} K_{2n}\backslash\{e\}$,( $l$ copies of $K_{2n}$), $n\geq 3,l\geq 2.$ Then, the spectrum of $\epsilon(G)$ consists of: 
 \begin{enumerate}
     \item $-1$ with multiplicity at least $n-3.$
     \item $-2$ with multiplicity at least $1.$
     \item $0$ with multiplicity  $(n-1)l.$
     \item $-2n$ with multiplicity  $l-1.$
     \item Roots of the polynomial $x^{3}+\big(-2nl+n+1\big)x^{2}+\big(2n+ln^{2}-2n^{2}-2\big)x+-4n+2ln^{2}.$
 \end{enumerate}
\end{lemma}
\begin{proof}
    By suitable labeling of vertices
 $$\epsilon(K_{2n}\circ_{n} K_{2n}\circ_{n}\cdots\circ_{n}K_{2n}\backslash\{e\})=\begin{pmatrix}
        (J-I)_{n-2\times n-2} &J_{n-2\times 2} &J_{n-2\times n}&J_{n-2\times n} &\cdots&J_{n-2\times n}\\
        J_{2\times n-2} & 2(J-I)_{2\times 2} & 0_{2\times n}&0_{2\times n}&\cdots&0_{2\times n}\\
        J_{n\times n-2}&0_{n\times 2}&0_{n\times n}&2J_{n\times n}&\cdots&2J_{n\times n}\\
        J_{n\times n-2}&0_{n\times 2}&2J_{n\times n}&0_{n\times n}&\cdots&2J_{n\times n}\\
        \vdots&\vdots&\vdots &\vdots&\ddots&\vdots\\
        J_{n\times n-2}&0_{n\times 2}&2J_{n\times n}&2J_{n\times n}&\cdots&0_{n\times n}
    \end{pmatrix}_{(l+1)n\times (l+1)n}$$
For $i=2,3\ldots, n-2,$ consider the vector 
$$x^{(i)}=(x_{1,1},x_{1,2},\ldots,x_{1,n-2},x_{2,1},x_{2,2},x_{3,1},x_{3,2},\ldots,x_{3,n},x_{4,1},x_{4,2},\ldots,x_{4,n},\cdots,x_{l+2,1},x_{l+2,2},\ldots,x_{l+2,n})$$
such that $x_{1,1}=1,x_{1,i}=-1,$ and $x_{s,t}\neq x_{1,1},x_{1,i}.$
Then, $\epsilon(G)x^{(i)}=(-1)x^{(i)},$ and all the vectors $x^{(i)}, 2\leq i\leq n-2$ are linearly independent. Thus, $-1$ is an eigenvalue of $\epsilon(K_{2n}\circ_{n} K_{2n}\cdots \circ K_{2n}\backslash\{e\})$ with multiplicity at least $n-3.$\\
Now, consider the vector
$$y=(y_{1,1},y_{1,2},\ldots,y_{1,n-2},y_{2,1},y_{2,2},y_{3,1},y_{3,2},\ldots,y_{3,n},y_{4,1},y_{4,2},\ldots,y_{4,n},\cdots,y_{l+2,1},y_{l+2,2},\ldots,y_{l+2,n})$$  
such that $y_{2,1}=1,y_{2,2}=-1$ and $y_{u,v}=0$ for $y_{u,v}\neq y_{2,1},y_{2,2}.$
Then, $\epsilon(K_{2n}\circ_{n} K_{2n}\circ_{n}\cdots \circ K_{2n}\backslash\{e\})y=-2y.$
Therefore, $-2$ is an eigenvalue of $\epsilon(K_{2n}\circ_{n} K_{2n}\circ_{n}\cdots \circ_{n} K_{2n}\backslash\{e\})$ with multiplicity at least $1.$\\

For $j=3,4,\ldots,l+2,$ consider the vector,
$$z^{(j,t)}=(z_{1,1},z_{1,2},\ldots,z_{1,n-2},z_{2,1},z_{2,2},z_{3,1},z_{3,2},\ldots,z_{3,n},z_{4,1},z_{4,2},\ldots,z_{4,n},\cdots,z_{l+2,1},z_{l+2,2},\ldots,z_{l+2,n})$$ 
such that $z_{j,1}=1,z_{j,t}=-1,$ for $2\leq t \leq n,$ and $z_{s,p}=0,$ for $z_{s,p}\neq z_{j,1},z_{j,t}.$ Then, $\epsilon(K_{2n}\circ_{n} K_{2n}\circ_{n}\cdots \circ_{n} K_{2n}\backslash\{e\})z^{(j,i)}=0,$ and all the vectors $z^{(j,i)}$ are linearly independent. Thus, $0$ is an eigenvalue of $\epsilon(K_{2n}\circ_{n} K_{2n}\circ_{n}\cdots \circ_{n} K_{2n}\backslash\{e\})$ with multiplicity at least $(n-1)l.$\\
Note that all the eigenvectors constructed so far are orthogonal to 
$$\begin{pmatrix}
    J_{n-2\times 1}\\
    0_{2\times 1}\\
    0_{n\times 1}\\
    \vdots\\
    0_{n\times 1}
\end{pmatrix},
\begin{pmatrix}
    0_{n\times 1}\\
    J_{2\times 1}\\
    0_{n-2\times 1}\\
    \vdots\\
    0_{n\times 1}
\end{pmatrix},
\begin{pmatrix}
    0_{n-2\times 1}\\
    0_{2\times 1}\\
    J_{n\times 1}\\
    \vdots\\
    0_{n\times 1}
\end{pmatrix},
\ldots
\begin{pmatrix}
    0_{n-2\times 1}\\
    0_{2\times 1}\\
    0_{n\times 1}\\
    \vdots\\
    J_{n\times 1}
\end{pmatrix}.$$
Since  $\epsilon(K_{2n}\circ_{n} K_{2n}\circ_{n}\cdots \circ_{n} K_{2n}\backslash\{e\})$ is a  symmetric matrix of order $(n+1),$ $\mathbb{R}^{(n+1)l}$ has an orthogonal basis consisting of eigenvectors of $\epsilon(K_{2n}\circ_{n} K_{2n}\circ_{n}\cdots \circ_{n} K_{2n}\backslash\{e\}).$
Hence, the remaining $l+2$ eigenvectors are spanned by these $l+2$ vectors.
 Thus, if $\sigma_{1}$ is an eigenvalue of $\epsilon(K_{2n}\circ_{n} K_{2n}\circ_{n}\cdots \circ_{n} K_{2n}\backslash\{e\})$ with an eigenvector $\zeta_{1},$ of the form $\begin{pmatrix}
    q_{1}J_{n-2\times 1}\\
    q_{2}J_{2\times 2}\\
    q_{3}J_{n\times n}\\
    q_{4}J_{n\times n}\\
    \vdots\\
    q_{l+2}J_{n\times n}
    \end{pmatrix},$
   where $q_{i}\in \mathbb{R},$ then from $\epsilon(K_{2n}\circ_{n} K_{2n}\circ_{n}\cdots \circ_{n} K_{2n}\backslash\{e\})\zeta_{1}=\sigma_{1}\zeta_{1},$ it follows that the remaining $l+2$ eigenvalues are the eigenvalues of the matrix 
     $$B_{4}=\begin{pmatrix}
            n-3&2&n&n&n&n&\cdots&n\\
            n-2&2&0&0&0&0&\cdots&0\\
            n-2&0&0&2n&2n&2n&\cdots&2n\\
            n-2&0&2n&0&2n&2n&\cdots&2n\\
            \vdots& \vdots& \vdots& \vdots& \vdots& \vdots&\cdots& \vdots\\
            n-2&0&2n&2n&2n&2n&\cdots&0
        \end{pmatrix}_{(l+2)\times (l+2)}.$$
 Now we will find the eigenvalues of the matrix  $B_{4}$.      
 Consider the vector $$u^{(q)}=(u_{1},u_{2},u_{3},u_{4},\ldots,u_{l+2})$$    
such that $u_{3}=1,u_{q}=-1,$ for $4\leq q\leq l+2,$ and $u_{r}=0,$ for $r\neq 3,q.$
Then, all the vectors $u^{(q)}$ are linearly independent and $B_{4}u^{(q)}=(-2n)u^{(q)}.$ Thus, $-2n$ is an eigenvalue of $B_{4}$ with multiplicity at least $l-1.$
 Consider the equitable quotient matrix of $B_{4}$ given by $$\tilde{B_{4}}=\begin{pmatrix}
    n-3&2&ln\\
    n-2&2&0\\
    n-2&0&2n(l-1)
\end{pmatrix}.$$
The characteristic polynomial of $\tilde{B_{4}}$ is, 
\begin{equation}\label{P(B4) for case 1, general}
P_{\tilde{B_{4}}}(x)=x^{3}+x^{2}(n-2nl+1)+x(2n+ln^{2}-2n^{2}-2)-4n+2ln^{2}.
\end{equation}
Now, $P_{\tilde{B_{4}}}(-2n)=2ln^{2}(1-5n)\neq 0.$
Thus, the eigenvalues of $B_{4}$ are $-2n$ with multiplicity  $l-1$ and roots of the polynomial (\ref{P(B4) for case 1, general}). The result follows from Theorem \ref{equitable quotient matrix eigenvalue thm}.
\end{proof}
\begin{corollary}\label{e spectrum of k coalencence l=2, case 1}
The eigenvalues of the eccentricity matrix of $K_{2n}\circ_{n} K_{2n}\backslash\{e\},n\geq 3$ are 
$-1, 0,-2n $ with multiplicities $n-3, 2n-2, 1$ respectively,$-2$ with multiplicity at least $1,$
and the remaining eigenvalues are roots of the cubic equation $x^{3}+(1-3n)x^{2}+2(n-1)x+4n(n-1).$
\end{corollary}
\begin{theorem}\label{case1 green edge deleting energy change , l=2}
Let $G=K_{2n}\circ_{n}K_{2n}\backslash \{e\}, n\geq 3.$
    Then, the eccentricity energy of $K_{2n}\circ_{n}K_{2n}\backslash \{e\}$ is greater than the eccentricity energy of $K_{2n}\circ_{n}K_{2n}.$
\end{theorem}
\begin{proof}
    Let $\alpha_{1}, \alpha_{2}, \alpha_{3}$ be the  roots of the polynomial $g_{\tilde{B}_{4}}(x)=x^{3}+(1-3n)x^{2}+2(n-1)x+4n(n-1).$ We have the following relation between roots and coefficients of $g_{\tilde{B}_{4}}(x):$
    \begin{align}\label{sumalphafor case1 l=2}
        \alpha_{1}+\alpha_{2}+\alpha_{3}=&3n-1,\\
\alpha_{1}\alpha_{2}+\alpha_{1}\alpha_{3}+\alpha_{2}\alpha_{3}=&2(n-1),\notag\\
        \alpha_{1}\alpha_{2}\alpha_{3}=&-4n(n-1).\notag
         \end{align}
Since, $\alpha_{1}\alpha_{2}\alpha_{3}<0,$   $g_{\tilde{B}_{4}}(x)$ has either one or three   negative roots. But, $\alpha_{1}+\alpha_{2}+\alpha_{3}>0.$ Thus, all $\alpha_{i}$s can not be negative.  Hence, $g_{\tilde{B}_{4}}(x)$  must have exactly one negative root and $2$ positive roots\\
Assume that $\alpha_{3}<0<\alpha_{2}\leq \alpha_{1}.$
Now, $g_{\tilde{B}_{4}}(-n)=n(-4n^{2}+3n-2)<0$ and $g_{\tilde{B}_{4}}(-1)=4n^{2}-9n+2>0.$ Thus, we have, $-n<\alpha_{3}<-1.$ \\
From equation (\ref{sumalphafor case1 l=2}),
\begin{equation} \label{eq:alpha-sum-squared for case 1,l=2}
(\alpha_{1}+\alpha_{2}+|\alpha_{3}|)^{2}=(3n-1)^{2}-2\alpha_{1}\alpha_{3}-2\alpha_{2}\alpha_{3}+2\alpha_{1}|\alpha_{3}|+2\alpha_{2}|\alpha_{3}|.
\end{equation}

We have, $$-\alpha_{3}>1.$$
Thus,$$3n-1-\alpha_{3}>3n-1+1.$$
Hence, \begin{equation}\label{step3 case 1 l=2}
(3n-1-\alpha_{3})|\alpha_{3}|>3n|\alpha_{3}|>3n>2n.
\end{equation}
Using equation (\ref{sumalphafor case1 l=2}), inequality (\ref{step3 case 1 l=2}) becomes, $$2n<(\alpha_{1}+\alpha_{2})|\alpha_{3}|.$$
Therefore, $$(1-3n)^{2}+8n<(3n-1)^{2}+4\alpha_{1}|\alpha_{3}|+4\alpha_{2}|\alpha_{3}|,$$
 $$(n-1)+2n+\sqrt{(1-3n)^{2}+8n}< (n-1)+2n+\sqrt{(3n-1)^{2}+4\alpha_{1}|\alpha_{3}|+4\alpha_{2}|\alpha_{3}|}. $$
Thus, by equation (\ref{eq:alpha-sum-squared for case 1,l=2}),
\begin{equation*}
(3n-1)+\sqrt{(1-3n)^{2}+8n}\leq (3n-1)+\alpha_{1}+\alpha_{2}+|\alpha_{3}|.
\end{equation*}
 By Corollaries \ref{energy k coalencence for all cases} and  \ref{e spectrum of k coalencence l=2, case 1},
 \begin{equation*}
 E_{\epsilon}(K_{2n}\circ_{n}K_{2n})\leq E_{\epsilon}(K_{2n}\circ_{n}K_{2n}\backslash \{e\}).
\end{equation*}

\end{proof}

\begin{theorem}\label{energy change general case 1. green edge.}
Let $G=K_{2n}\circ_{n} K_{2n}\circ_{n}\cdots \circ_{n} K_{2n},$ ($l$ copies of $K_{2n}$) $n\geq 3,$ $l\geq 3.$ Then,
$E_{\epsilon}(G)\leq E_{\epsilon}(G\backslash\{e\}).$    
\end{theorem}
\begin{proof}
    Let $\beta_{1},\beta_{2},$ and $\beta_{3}$ be the roots of  the polynomial (\ref{P(B4) for case 1, general}). We have the following relations between the roots and coefficients of (\ref{P(B4) for case 1, general}):\\
    \begin{align}\label{relation between roots and coefficients in general l, case1}
        \beta_{1}+\beta_{2}+\beta_{3}=&2nl-n-1\\
        \beta_{1}\beta_{2}+ \beta_{1}\beta_{3}+ \beta_{2}\beta_{3}=&2n+ln^{2}-2n^{2}-2\notag\\
         \beta_{1}\beta_{2}\beta_{3}=&2n(2-ln)<0.\notag
    \end{align}
    Since $\beta_{1}\beta_{2}\beta_{3}=2n(2-ln)<0,$ the polynomial $P_{\tilde{B^{4}}}(x)$ has either one or three negative roots. As $\beta_{1}+\beta_{2}+\beta_{3}=2nl-n-1>0,$ all $\beta_{i}$'s can not be negative. Thus, $P_{\tilde{B_{4}}}(x)$ has exactly one negative and $2$ positive roots. Assume, that $\beta_{3}<0<\beta_{2}\leq \beta_{1}.$
    Now, \begin{align*}
        |\beta_{1}|+ |\beta_{2}|+ |\beta_{3}|=&\beta_{1}+\beta_{2}+|\beta_{3}|.
    \end{align*}
   Then, by (\ref{relation between roots and coefficients in general l, case1}) 
    \begin{align}\label{case1 general case, sum of mod alpha i's}
        |\beta_{1}|+ |\beta_{2}|+ |\beta_{3}|=&2nl-n-1+2|\beta_{3}|.
    \end{align}
  By Corollary\ref{energy k coalencence for all cases},
 \begin{align*}
   E_{\epsilon}(G)=&2\big(n(2l-1)-1\big)\\
                =&(n-1)+2n(l-1)+\big(n(2l-1)-1\big)\\
                <&(n-1)+2n(l-1)+\big(n(2l-1)-1\big)+2|\beta_{3}|.\\
                \end{align*}
  Thus, by \ref{case1 general case, sum of mod alpha i's},              
\begin{align*}
   E_{\epsilon}(G)<&(n-1)+2n(l-1)+|\beta_{1}|+|\beta_{2}|+|\beta_{3}|.\\
                \end{align*}
By Lemma \ref{spectrum of coalesence, general, case 1},
$$E_{\epsilon}(G)<E_{\epsilon}(G\backslash\{e\}).$$
\end{proof}
From Theorems \ref{case1 green edge deleting energy change , l=2} and \ref{energy change general case 1. green edge.}, 
we obtain the following result.
\begin{theorem}\label{final conclusion for case 1 green edge}
    Let $G=K_{2n}\circ_{n} K_{2n}\circ_{n}\cdots \circ_{n} K_{2n},$ (l copies of $K_{2n}$) $n\geq 3,$ $l\geq 2.$ Then,
$E_{\epsilon}(G)\leq E_{\epsilon}(G\backslash\{e\}).$  
\end{theorem}

\subsection{Eccentricity energy of  \texorpdfstring{$K_{2n}\circ_{n} K_{2n}\circ_{n}\cdots \circ_{n}K_{2n}\backslash\{e\}$} ffor case 2 }
This section analyzes the eccentricity eigenvalues of
$G=K_{2n}\circ_{n} K_{2n}\circ_{n}\cdots \circ_{n} K_{2n}\backslash\{e\}$  and compares the resulting eccentricity energy with that of the original graph.

\begin{lemma} \label{e spectrum of k coalecence general, case 2}
    Let $G=K_{2n}\circ_{n} K_{2n}\circ_{n}\cdots \circ_{n} K_{2n}\backslash\{e\}$  ($l$ copies of $K_{2n}$) $,n\geq 3, l\geq 2.$ Then, the spectrum of $\epsilon(G)$ consists of: 
 \begin{enumerate}
     \item $-1$ with multiplicity at least $n-2.$
     \item $0$ with multiplicity at least $(n-1)(l-1)+n-2.$
     \item $-2n$ with multiplicity al least $l-2.$
     \item Roots of the polynomial $x^{5}+\big(-2nl+3n+2     \big)x^{4}+\big( -3nl+7n-3ln^{2}-3 \big)x^{3}+\big( -12n+6ln-4ln^{2}+2ln^{3}+4n^{2}-4n^{3}-4    \big)x^{2}+\big( -12n-4ln+8ln^{2}+4ln^{3}-4n^{2}-4n^{3}+4    \big)x+-8ln^{3}+32ln^{2}-24ln+16n^{3}-48n^{2}+32n.$
 \end{enumerate}
\end{lemma}
\begin{proof}
    By suitable labelling of vertices 
$$\epsilon(K_{2n}\circ_{n} K_{2n}\circ_{n}\cdots\circ_{n}K_{2n}\backslash\{e\})=\begin{pmatrix}
        0_{1\times 1} & J_{1\times n-1} &2_{1\times 1}&0_{1\times n-1}&0_{1\times n}&\cdots&0_{1\times n}\\
        J_{n-1\times 1}&(J-I)_{n-1\times n-1}&J_{n-1\times 1}&J_{n-1\times n-1}&J_{n-1\times n}&\cdots&J_{n-1\times n}\\
        2_{1\times 1}&J_{1\times n-1}&0_{1\times 1}&0_{1\times n-1}&2J_{1\times n}&\cdots&2J_{1\times n}\\
        0_{n-1\times 1}&J_{n-1\times n-1}&0_{n-1\times 1}&0_{n-1\times n-1}&2J_{n-1\times n}&\cdots&2J_{1\times n}\\
        0_{n\times 1}&J_{n\times n-1}&2J_{n\times 1}&2J_{n\times n-1}&0_{n\times n}&\cdots&2J_{1\times n}\\
        \vdots&\vdots&\vdots&\vdots&\vdots&\ddots&\vdots\\
        0_{n\times1}&J_{n\times n-1}&2J_{n\times1}&2J_{n\times n-1}&2J_{n\times n}&\cdots&0_{n\times n}
    \end{pmatrix}_{n(l+1)\times n(l+1)}.$$
For $i=2,3,\ldots,n-1,$ consider the vector $$x^{(i)}=(x_{1,1},x_{2,1},x_{2,2},\ldots,x_{2,n-1},x_{3,1},x_{4,1},x_{4,2},\ldots,x_{4,n-1},x_{5,1},x_{5,2},\ldots,x_{5,n},\cdots x_{l+3,1},x_{l+3,2},\ldots,x_{l+3,n})$$ 
such that $x_{2,1}=1,$ $x_{2,i}=-1$  and $x_{u,v}=0,$ for $x_{u,v}\neq x_{2,1}, x_{2,i}.$ Then, the vectors $x^{(i)}$s are linearly independent, and $\epsilon(K_{2n}\circ_{n} K_{2n}\circ_{n} \cdots\circ_{n} K_{2n}\backslash\{e\})x^{(i)}=(-1)x^{(i)}.$ Thus, $-1$ is an eigenvalue of $\epsilon(K_{2n}\circ_{n} K_{2n}\circ_{n} \cdots\circ_{n} K_{2n}\backslash\{e\})$ with multiplicity at least $n-2.$\\

\noindent For $ t=2,3,\ldots, n-1,$ consider the vector,
$$y^{(4,t)}=(y_{1,1},y_{2,1},y_{2,2},\ldots,y_{2,n-1},y_{3,1}, y_{4,1}, y_{4,2},\ldots, y_{4,n-1},y_{5,1},y_{5,2},\ldots,y_{5,n},\cdots ,y_{l+3,1},y_{l+3,2},\ldots,y_{l+3,n})$$ 
such that $y_{4,1}=1,$ $y_{4,t}=-1$  and $y_{p,q}=0,$ for $y_{p,q}\neq y_{4,1}, y_{4,t}.$ \\

For $j=5,6,\ldots,l+3,$
consider the vector $$z^{(j,s)}=(z_{1,1},z_{2,1},z_{2,2},\ldots,z_{2,n-1},z_{3,1},z_{4,1},z_{4,2},\ldots,z_{4,n-1},z_{5,1},z_{5,2},\ldots,z_{5,n},\cdots z_{l+3,1},z_{l+3,2},\ldots,z_{l+3,n})$$ 
such that $z_{j,1}=1,$ $z_{j,s}=-1,$ for $2\leq s\leq n,$ and $z_{r,f}=0,$ for $z_{r,f}\neq z_{j,1}, z_{j,i}.$ 
Then, $\epsilon(K_{2n}\circ_{n} K_{2n}\circ_{n}\cdots\circ_{n} K_{2n}\backslash\{e\})y^{(j,i)}=0=\epsilon(K_{2n}\circ_{n} K_{2n}\circ_{n} \cdots \circ_{n} K_{2n}\backslash\{e\})y^{(j,i)},$
and the vectors $y^{(j,i)},$ $y^{(4,i)}$ are linearly independent. Thus,  $0$ is an eigenvalue of $\epsilon(K_{2n}\circ_{n} K_{2n}\circ_{n} \cdots\circ_{n} K_{2n}\backslash\{e\})$ with multiplicity atleast $(n-1)(l-1)+n-2.$\\
Note that all the eigenvectors constructed so far are orthogonal to 
$$\begin{pmatrix}
    J_{1\times 1}\\
    0_{n-1\times 1}\\
    0_{1\times 1}\\
    0_{n-1\times 1}\\
    0_{n\times 1}\\
    \vdots\\
    0_{n\times 1}
\end{pmatrix},
\begin{pmatrix}
    0_{1\times 1}\\
    J_{n-1\times 1}\\
    0_{1\times 1}\\
    0_{n-1\times 1}\\
     0_{n\times 1}\\
    \vdots\\
    0_{n\times 1}
\end{pmatrix},
\begin{pmatrix}
    0_{1\times 1}\\
    0_{n-1\times 1}\\
    J_{1\times 1}\\
    0_{n-1\times 1}\\
     0_{n\times 1}\\
    \vdots\\
    0_{n\times 1}
\end{pmatrix}
\cdots
\begin{pmatrix}
    0_{1\times 1}\\
    0_{n-1\times 1}\\
    0_{1\times 1}\\
    0_{n-1\times 1}\\
     0_{n\times 1}\\
    \vdots\\
    J_{n\times 1}
\end{pmatrix}.$$
Since $\epsilon(K_{2n}\circ_{n} K_{2n}\circ_{n}\cdots\circ_{n} K_{2n}\backslash\{e\})$ is a  symmetric matrix of order $n(l+1)$, $\mathbb{R}^{n(l+1)}$ has an orthogonal basis consisting of eigenvectors of $\epsilon(K_{2n}\circ_{n} K_{2n}\circ_{n} \cdots\circ_{n} K_{2n}\backslash\{e\}).$ 
Hence, the remaining $l+3$ vectors are spanned by these $l+3$ vectors.
Thus, if $\sigma_{2}$ is an eigenvalue of  $\epsilon(K_{2n}\circ_{n} K_{2n}\circ_{n}\cdots\circ_{n} K_{2n}\backslash\{e\})$ with an eigenvector $\zeta_{2},$  of the form
$\begin{pmatrix}
    c_{1}J_{1\times 1}\\
    c_{2}J_{n-1\times 1}\\
    c_{3}J_{1\times 1}\\
    c_{4}J_{n-1\times 1}\\
    \vdots\\
    c_{l+3}J_{n\times 1}
\end{pmatrix},$ 
where $c_{i}\in \mathbb{R},$ then from  $\epsilon(K_{2n}\circ_{n} K_{2n}\circ_{n} \cdots\circ_{n} K_{2n}\backslash\{e\})\zeta_{2}=\sigma_{2}\zeta_{2}, $ it follows  that the remaining $l+3$ eigenvalues are eigenvalues of the matrix,
$$B_{5}=\begin{pmatrix}
        0 &n-1&2&0&0&0&0&\cdots&0\\
        1&n-2&1&n-1&n&n&n&\cdots&n\\
        2&n-1&0&0&2n&2n&2n&\cdots&2n\\
        0&n-1&0&0&2n&2n&2n&\cdots&2n\\
        0&n-1&2&2(n-1)&0&2n&2n&\cdots&2n\\
        0&n-1&2&2(n-1)&2n&0&2n&\cdots&2n\\
        0&n-1&2&2(n-1)&2n&2n&0&\cdots&2n\\
        \vdots&\vdots&\vdots&\vdots&\vdots&\vdots&\vdots&\ddots&\vdots\\
        0&n-1&2&2(n-1)&2n&2n&0&\cdots&0\\
    \end{pmatrix}_{l+3\times l+3}.$$    
 Now, for $h=5,6,\ldots l+3,$  consider the vector $$w^{(h)}=(w_{1},w_{2},w_{3},w_{4},w_{5},\ldots,w_{l+3})$$ 
 such that $w_{5}=1,$ $w_{h}=-1$ and $w_{d}=0$ for $d\neq 5,h.$
 Then, $B_{6}w^{(h)}=(-2n)w^{(h)}.$ Since all $w^{(h)}$s are linearly independent, $-2n$ is an eigenvalue of $B_{5}$ with multiplicity at least $l-2.$ 

 The equitable quotient matrix of $B_{5}$ is $\tilde{B_{5}}=\begin{pmatrix}
     0 &n-1&2&0&0\\
        1&n-2&1&n-1&n(l-1)\\
        2&n-1&0&0&2n(l-1)\\
        0&n-1&0&0&2n(l-1)\\
        0&n-1&2&2(n-1)&2n(l-2)
       \end{pmatrix}.$
The characteristic polynomial of $\tilde{B_{5}}$ is
  $P_{\tilde{B{5}}}(x)=x^{5}+\big(-2nl+3n+2     \big)x^{4}+\big( -3nl+7n-3ln^{2}-3 \big)x^{3}+\big( -12n+6ln-4ln^{2}+2ln^{3}+4n^{2}-4n^{3}-4    \big)x^{2}+\big( -12n-4ln+8ln^{2}+4ln^{3}-4n^{2}-4n^{3}+4    \big)x+-8ln^{3}+32ln^{2}-24ln+16n^{3}-48n^{2}+32n.$
Moreover, $P_{\tilde{B{5}}}(-2n)=8n(5n-3)(l-1)>0.$
Therefore, the eigenvalues of $B_{5}$ are $-2n$ with multiplicity $l-2$ and roots of the polynomial $P_{\tilde{B{5}}}(x).$ The result follows from \ref{equitable quotient matrix eigenvalue thm}.

  
\end{proof}
\begin{corollary}\label{case 2 blue edge deleting spectrum}
    The eigenvalues of the eccentricity matrix of $K_{2n}\circ_{n} K_{2n}\backslash\{e\}$ are $-1$ with multiplicity $n-2,$ $0$ with multiplicity $2n-3,$ and the remaining eigenvalues are roots of the  equation $x^{5}+(2-n)x^{4}+(-6n^{2}+n-3)x^{3}+(-4n^{2}-4)x^{2}+(4n^{3}+12n^{2}-20n+4)x+16n(n-1)=0.$
\end{corollary}
\begin{theorem}\label{case2 blue edge deleting energy change , l=2}
Let $G=K_{2n}\circ_{n}K_{2n}\backslash \{e\},n\geq 3.$
    Then, eccentricity energy of $K_{2n}\circ_{n}K_{2n}\backslash \{e\}$ is greater than the eccentricity energy of $K_{2n}\circ_{n}K_{2n}.$
\end{theorem}
\begin{proof}
 Let $\alpha_{1},\alpha_{2},\alpha_{3},\alpha_{4},\alpha_{5}$ be the roots of the polynomial, $g_{\tilde{B}_{5}}(x)=x^{5}+(2-n)x^{4}+(-6n^{2}+n-3)x^{3}+(-4n^{2}-4)x^{2}+(4n^{3}+12n^{2}-20n+4)x+16n(n-1).$ We have the following relation between the roots and coefficients of $g_{\tilde{B}_{5}}(x):$
\begin{align}\label{sumalpha 1 to 5}
\alpha_{1}+\alpha_{2}+\alpha_{3}+\alpha_{4}+\alpha_{5}=&n-2, \\
\sum_{1\leq i<j\leq 5}\alpha_{i}\alpha_{j}=&-6n^{2}+n-3, \notag\\
\sum_{1\leq i<j<k\leq 5}\alpha_{i}\alpha_{j}\alpha_{k}=&4n^{2}+4,  \notag\\
\sum_{1\leq i<j<k<l\leq 5}\alpha_{i}\alpha_{j}\alpha_{k}\alpha_{l}=&4n^{3}+12n^{2}-20n+4,  \notag\\
\alpha_{1}\alpha_{2}\alpha_{3}\alpha_{4}\alpha_{5}=&-16n(n-1).\notag
\end{align}
Since $\alpha_{1}\alpha_{2}\alpha_{3}\alpha_{4}\alpha_{5}<0,$ $g_{\tilde{B}_{5}}(x)$ has either one, three or five negative roots. As, $\alpha_{1}+\alpha_{2}+\alpha_{3}+\alpha_{4}+\alpha_{5}>0,$ all $\alpha_{i}$ cannot be negative. Note that, $g_{\tilde{B}_{5}}(x)$ has only two sign changes. Thus, $g_{\tilde{B}_{5}}(x)$ can have atmost two positive roots. Therefore, $g_{\tilde{B}_{5}}(x)$ has three negative roots.\\
Let $\alpha_{5}\leq \alpha_{4}\leq \alpha_{3}<0<\alpha_{2}\leq \alpha_{1}.$  Also we have, \begin{align*}
    g_{\tilde{B}_{5}}(-2n)=&8n(5n-3)>0,\\
    g_{\tilde{B}_{5}}(-2n-1)=&2(-20n^{4}-18n^{3}+30n^{2}+n-2)<0,\\
    g_{\tilde{B}_{5}}(-2)=&8n^{2}(3-n)<0,\\
    g_{\tilde{B}_{5}}(0)=&16n(n-1)>0.
    \end{align*}
Thus, by the intermediate value theorem,
$$-2n-1<\alpha_{5}<-2n<\alpha_{4}<-2<\alpha_{3}<0.$$
Hence, \begin{equation}\label{lowerbound for sum alpha3alpha4alpha5}
|\alpha_{3}|+|\alpha_{4}|+|\alpha_{5}|\geq 2n+2.
\end{equation}
Now, by Lemma \ref{case 2 blue edge deleting spectrum},
\begin{align*}
    E_{\epsilon}(K_{2n}\circ_{n}K_{2n}\backslash \{e\})=& n-2+\alpha_{1}+\alpha_{2}+|\alpha_{3}|+|\alpha_{4}|+|\alpha_{5}|.
\end{align*}
Using (\ref{sumalpha 1 to 5}), (\ref{lowerbound for sum alpha3alpha4alpha5}), and Corollary \ref{energy k coalencence for all cases},
\begin{align*}
E_{\epsilon}(K_{2n}\circ_{n}K_{2n}\backslash \{e\})=&2n-4+2(|\alpha_{3}|+|\alpha_{4}|+|\alpha_{5}|),   
\end{align*}
\begin{align*}
E_{\epsilon}(K_{2n}\circ_{n}K_{2n}\backslash \{e\})\geq& 2n-4+2(2n+2),\\
                                                  =&6n,\\
                                                  >&(3n-1)+\sqrt{(1-3n)^{2}+8n},
\end{align*}
\begin{align*}
   E_{\epsilon}(K_{2n}\circ_{n}K_{2n}\backslash \{e\})> E_{\epsilon}(K_{2n}\circ_{n}K_{2n}).
\end{align*}
\end{proof}

\begin{theorem} \label{energy change general case 2. blue edge.}
Let $G=K_{2n}\circ_{n} K_{2n}\circ_{n}\cdots \circ_{n} K_{2n}$ ($l$ copies of $K_{2n}$), $n\geq 5,$ $l\geq4.$ Then,
$E_{\epsilon}(G)\leq E_{\epsilon}(G\backslash\{e\}).$       
\end{theorem}
\begin{proof}
    Let $\beta_{1},\beta_{2},\beta_{3},\beta_{4}$ and $\beta_{5}$ be the roots of the polynomial $P_{\tilde{B}_{5}}(x)= x^{5}+\big(-2nl+3n+2     \big)x^{4}+\big( -3nl+7n-3ln^{2}-3 \big)x^{3}+\big( -12n+6ln-4ln^{2}+2ln^{3}+4n^{2}-4n^{3}-4    \big)x^{2}+\big( -12n-4ln+8ln^{2}+4ln^{3}-4n^{2}-4n^{3}+4    \big)x+-8ln^{3}+32ln^{2}-24ln+16n^{3}-48n^{2}+32n.$
    For $n\geq 5,l\geq 4,$
\begin{align}\label{relation between roots and coefficients of poly,for l>=4 case 2.}
 \beta_{1}+\beta_{2}+\beta_{3}+\beta_{4}+\beta_{5}=&(2l-3)n-2>0, \\
 \beta_{1}\beta_{2}\beta_{3}\beta_{4}\beta_{5}=&8n\big( l(n^{2}-4n+3)-2(n^{2}-3n+2)\big),\notag\\
 =&8n\big(n^{2}(l-2)+n(6-4l)+3l-4 \big),\notag\\
 >&0.\notag
\end{align}
 Since, $\beta_{1}\beta_{2}\beta_{3}\beta_{4}\beta_{5}>0,$ $P_{\tilde{B}_{5}}(x)$ has either zero, two, or four negative roots. But for $l\geq 4, n\geq 5$ we have, $\sum\beta_{i}\beta_{j}=-3ln-3ln^{2}+7n-3<0.$ Therefore,  $P_{\tilde{B}_{5}}(x),$ will have atleast one  negative root. 
 Also, for $n\geq 5, l\geq 4,$ $P_{\tilde{B}_{5}}(-x)$ has two sign changes. Hence, $P_{\tilde{B}_{5}}(x)$ can have at most two negative roots.
 Thus,  $P_{\tilde{B}_{5}}(x)$ has two negative roots and three positive roots.\\
 Let $\beta_{5}\leq \beta_{4}<0<\beta_{3}\leq\beta_{2}\leq \beta_{1}.$
 We also have, \begin{align*}
   P_{\tilde{B}_{5}}(-n^{2})=&(-n^{10}+2n^{8}+3n^{6}+3n^{9})+(-2ln^{9}+3ln^{8}+5ln^{7})+(-11n^{7}+28n^{3})+\\&(-8n^{5})+(-4ln^{6}+4n^{6}+2ln^{5})+(-4ln^{3}-52n^{2})+(32ln^{2}-8ln^{4})+(-24ln+32n)\\< & 0,\\
  P_{\tilde{B}_{5}}(-2n)=&8n(5n-3)(l-1)>0,
  \end{align*}
    and
  $$P_{\tilde{B}_{5}}(-1)=(n-1)\big(16n^{2}+13ln-10ln^{2}+(4-24n)\big)<0.$$
Therefore,
\begin{equation*}
   -n^{2}<\beta_{5}<-2n<\beta_{4}<-1,
\end{equation*}
\begin{align*}
    2n+1&<-\beta_{4}-\beta_{5},
\end{align*}
\begin{align*}
 4n+2&<-2(\beta_{4}+\beta_{5}),
    \end{align*}
    \begin{align*}
    n&<-3n-2-2(\beta_{4}+\beta_{5}),
    \end{align*}
\begin{align*}
    n+2nl<&2nl-3n-2-2(\beta_{4}+\beta_{5}),\\
   2\big(n(2l-1)-1\big)<&(n-2)+2n(l-2)+2nl-3n-2-(\beta_{4}+\beta_{5}).\\
   \end{align*}
   From (\ref{relation between roots and coefficients of poly,for l>=4 case 2.}) we obtain,
\begin{align*}
  2\big(n(2l-1)-1\big)<&(n-2)+2n(l-2)+\beta_{1}+\beta_{2}+\beta_{3}-\beta_{4}-\beta_{5}.\\  
\end{align*}
Then, by Corollary \ref{energy k coalencence for all cases} and Lemma \ref{e spectrum of k coalecence general, case 2},
$$E_{\epsilon}(K_{2n}\circ_{n} K_{2n}\circ_{n}\cdots \circ_{n} K_{2n})<E_{\epsilon}(K_{2n}\circ_{n} K_{2n}\circ_{n}\cdots \circ_{n} K_{2n}\backslash\{e\}).$$

\end{proof}

\begin{remark}\label{remark general case of omited case in case 2 blue edge}
Using the same method as in the preceding theorem, it follows that   for $l=3,n\geq 6,$  
$$E_{\epsilon}(K_{2n}\circ_{n}K_{2n}\circ_{n}K_{2n})=10n-2<E_{\epsilon}(K_{2n}\circ_{n}K_{2n}\circ_{n}K_{2n}\backslash\{e\}).$$
\end{remark}

\begin{note} \label{note general case of omited case in case 2 blue edge}
  For $l=3,n=3$ we have,  
  $$E_{\epsilon}(K_{6}\circ_{3}K_{6}\circ_{3}K_{6})=28.422<30.9233=E_{\epsilon}(K_{6}\circ_{3}K_{6}\circ_{3}K_{6}\backslash\{e\}).$$
  For $l=3,n=4$
  $$E_{\epsilon}(K_{8}\circ_{4}K_{8}\circ_{4}K_{8})=38<40.9698=E_{\epsilon}(K_{8}\circ_{4}K_{8}\circ_{4}K_{8}\backslash\{e\}),\\$$
 and for $l=3,n=5,$ 
 $$E_{\epsilon}(K_{10}\circ_{5}K_{10}\circ_{5}K_{10})=48<51.0058=E_{\epsilon}(K_{10}\circ_{5}K_{10}\circ_{5}K_{10}\backslash\{e\}).$$
 Similerly, for $l=4, n=3,$
 $$E_{\epsilon}(K_{6}\circ_{3}K_{6}\circ_{3}K_{6}\circ_{3}K_6)=40<42.618=E_{\epsilon}(K_{6}\circ_{3}K_{6}\circ_{3}K_{6}\circ_{3}K_{6}\backslash\{e\}),$$
 and for $l=4,n=4,$
 $$E_{\epsilon}(K_{8}\circ_{4}K_{8}\circ_{4}K_{8}\circ_{4}K_{8})=54<56.7213=E_{\epsilon}(K_{8}\circ_{4}K_{8}\circ_{4}K_{8}\circ_{4}K_{8}\backslash\{e\}).$$
\end{note}
From Theorems \ref{energy change general case 2. blue edge.}
, \ref{case2 blue edge deleting energy change , l=2}, \ref{remark general case of omited case in case 2 blue edge} and \ref{note general case of omited case in case 2 blue edge}, we obtain the following result.

we get the following result.
\begin{theorem}\label{final conclusion for case 2 blue edge}
    Let $G=K_{2n}\circ_{n} K_{2n}\circ_{n}\cdots \circ_{n} K_{2n},$ (l copies of $K_{2n}$) $n\geq 3,$ $l\geq 2.$ Then,
$E_{\epsilon}(G)\leq E_{\epsilon}(G\backslash\{e\}).$  
\end{theorem}

\subsection{Eccentricity energy of  \texorpdfstring{$K_{2n}\circ_{n} K_{2n}\circ_{n}\ldots \circ_{n}K_{2n}\backslash\{e\}$} ffor case 3 }
In this section, we study the eccentricity eigenvalues of the graph $G=K_{2n}\circ_{n} K_{2n}\circ_{n}\cdots \circ_{n} K_{2n}\backslash\{e\}$
and we compare the corresponding eccentricity energy with that of the original graph.

\begin{lemma}\label{e(G-e)is irreducible for case 3 edge, l=2}
Let $G=K_{a_{1}}\circ_{k}K_{a_{2}}\backslash\{e\}, a_{i}\geq 3,i=1,2.$ Then,   $\epsilon(K_{a_{1}}\circ_{k} K_{a_{2}}\backslash\{e\})$ is irreducible.
\end{lemma}
\begin{proof}
  Let  $$V(G)=V(K_{k})\cup V(K_{a_{1}}\backslash K_{k})\cup V(K_{a_{2}}\backslash K_{k})=V(G^{e}),$$ 
  where $K_{k}$ denotes the complete subgraph identified in the complete graphs $K_{a_{1}}$ and $K_{a_{2}}.$
  For any vertex $v \in V(K_{k}),$ we have $e(v)=1.$ \\
  Consequently, for any $v_{1},v_{2}\in V(K_{k}),$
  $$d(v_{1},v_{2})=1=min\{e(v_{1}),e(v_{2})\}.$$
  Thus, the vertex set $V(K_{k})$ induces a complete subgraph in $G^{e}.$\\
  Now, let $u\in V(G)\backslash V(K_{k}).$
  Then,
  $$d(u,v)=1=min\{e(u),e(v)\}.$$
  Hence, each vertex outside $V(K_{k})$ is adjacent to every vertex of $V(K_{k})$ in $G^{e}.$ It follows that  the eccentric graph $G^{e}$ is connected. By Lemma \ref{eccentric graph},
  it follows that the eccentricity matrix $\epsilon(G)$ is irreducible.
\end{proof}

\begin{lemma}\label{e spectrum of coalecence , general case, l>4, case 3}
Let $G=K_{2n}\circ_{n} K_{2n}\circ_{n}\cdots \circ_{n} K_{2n}\backslash\{e\}$($l$ copies of $K_{2n}$)  $,n\geq 3,l\geq 2.$ Then, the spectrum of $\epsilon(G)$ consists of: 
 \begin{enumerate}
     \item $-1$ with multiplicity at least $n-1.$
     \item $-2$ with multiplicity at least $1.$
     \item $0$ with multiplicity at least $(n-3)+(n-1)(l-1).$
      \item $-2n$ with multiplicity at least $l-2.$
     \item Roots of the polynomial $x^{4}-\big( 2nl-3n+1  \big)x^{3}-\big( 2n-2nl+3ln^{2}+2  \big)x^{2}+\big( 4n-12ln+2ln^{2}+2ln^{3}+4n^{2}-4n^{3}  \big)x+16n-24n^{2}+8n^{3}-16ln+16ln^{2}-4ln^{3}.$
  \end{enumerate}  
\end{lemma}
\begin{proof}
By suitable labelling of vertices $$\epsilon(G)=\begin{pmatrix}
        (J-I)_{n\times n} &J_{n\times 2} &J_{n\times n-2}&J_{n\times n}&J_{n\times n}&\cdots&J_{n\times n}\\
        J_{2\times n} & 2(J-I)_{2\times 2} & 0_{2\times n-2}&2J_{2\times n}&2J_{2\times n}&\cdots&2J_{2\times n}\\
        J_{n-2\times n}&0_{n-2\times 2}&0_{n-2\times n-2}&2J_{n-2\times n}&2J_{n-2\times n}&\cdots&2J_{n-2\times n}\\
        J_{n\times n}&2J_{n\times 2}&2J_{n\times n-2}&0_{n\times n}&2J_{n\times n}&\cdots &2J_{n\times n}\\
        \vdots&\vdots&\vdots&\vdots&\vdots&\ddots&\vdots\\
        J_{n\times n}&2J_{n\times 2}&2J_{n\times n-2}&2J_{n\times n}&2J_{n\times n}&\cdots&0_{n\times n}
    \end{pmatrix}_{n(l+1)\times n(l+1)}.$$
For $i=2,3,\ldots, n,$ consider the vector 
\begin{align*}
x^{(i)}=(&x_{1,1},x_{1,2},x_{1,3},\ldots,x_{1,n},x_{2,1},x_{2,2},x_{3,1},x_{3,2},\ldots,x_{3,n-2},
x_{4,1},x_{4,2},\ldots,x_{4,n},x_{5,1},x_{5,2},\ldots,x_{5,n},\cdots,\\
&x_{l+2,1},x_{l+2,2},\ldots,x_{l+2,n})
\end{align*}
such that $x_{1,1}=1,$ $x_{1,i}=-1$  and $x_{q,p}=0,$ for $x_{q,p}\neq x_{1,1}, x_{1,i}.$ Then, $x^{(i)}$s are linearly independent and $\epsilon(K_{2n}\circ_{n} K_{2n}\circ_{n}\cdots\circ_{n} K_{2n}\backslash\{e\})x^{(i)}=(-1)x^{(i)}.$ Thus, $-1$ is an eigenvalue of $\epsilon(K_{2n}\circ_{n} K_{2n}\circ_{n} \cdots\circ_{n} K_{2n}\backslash\{e\})$ with multiplicity at least $n-1.$\\

Now, consider the vector 
\begin{align*}
y=(&y_{1,1},y_{1,2},y_{1,3},\ldots,y_{1,n},y_{2,1},y_{2,2},y_{3,1},y_{3,2},\ldots,y_{3,n-2},y_{4,1},y_{4,2},\ldots,y_{4,n},y_{5,1},y_{5,2},\ldots,y_{5,n},\cdots \\ & y_{l+2,1},y_{l+2,2},\ldots,y_{l+2,n})
\end{align*}
such that $y_{2,1}=1,$ $y_{2,2}=-1$ and $y_{w,g}=0,$ for $y_{w,g}\neq y_{2,1}, y_{2,2}.$ Then,  $\epsilon(K_{2n}\circ_{n} K_{2n}\circ_{n} \cdots\circ_{n} K_{2n}\backslash\{e\})y=-2y.$ Thus, $-2$ is an eigenvalue of $\epsilon(K_{2n}\circ_{n} K_{2n}\circ_{n} \cdots\circ_{n} K_{2n}\backslash\{e\})$ with multiplicity at least $1.$\\

  For $j=2,3,\ldots,n-2$ consider the vector 
\begin{align*}
u^{(3,j)}=&u_{1,1},u_{1,2},\ldots,u_{1,n},u_{2,1},u_{2,2},u_{3,1},u_{3,2},\ldots,u_{3,n-2},u_{4,1},u_{4,2},\ldots,u_{4,n},u_{5,1},u_{5,2},\ldots,u_{5,n}\cdots\\& u_{l+3,1},u_{l+3,2},\ldots z_{l+3,n})
\end{align*}
such that $u_{3,1}=1,$ $u_{3,i}=-1$  and $u_{a,b}=0,$ for $u_{a,b}\neq u_{3,1}, u_{3,i}.$
Then,
$\epsilon(K_{2n}\circ_{n} K_{2n}\circ_{n} \cdots\circ_{n} K_{2n}\backslash\{e\}) u_{4,i}=0,$ and the vectors $u^{(3,i)}$ are linearly independent. Thus, $0$ is an eigenvalue of $\epsilon(G)$ with multiplicity at least $n-3.$

Furthermore, for $s=4,5,\ldots, l+2,$ consider the vector 
$$z^{(s,t)}=z_{1,1},z_{1,2},\ldots,z_{1,n},z_{2,1},z_{2,2},z_{3,1},z_{3,2},\ldots,z_{3,n-2},z_{4,1},z_{4,2},\ldots,z_{4,n},\cdots z_{l+2,1},z_{l+2,2},\ldots z_{l+2,n})$$
such that $z_{s,1}=1,z_{s,t}=-1,$ for $2\leq t\leq n,$ and $z_{c,d}=0,$ for all $z_{c,d}\neq z_{s,1}, z_{s,t}. $ Then, $\epsilon(K_{2n}\circ_{n} K_{2n}\circ_{n} \cdots\circ_{n} K_{2n}\backslash\{e\}) z^{(s,t)}=0. $ 
Also, all the vectors $z^{(s,t)}$ are linearly independent from each other and independent of the vectors $u^{(3,j)},j=2,3,\ldots,n-2$. Thus, $0$ is an eigenvalue of $\epsilon(K_{2n}\circ_{n} K_{2n}\circ_{n} \cdots\circ_{n} K_{2n}\backslash\{e\})$ with multiplicity at least  $(n-3)+(n-1)(l-1).$\\
 We note that all the eigenvectors constructed so far are orthogonal to
$$\begin{pmatrix}
    J_{n\times 1}\\
    0_{2\times 1}\\
    0_{n-2\times 1}\\
     0_{n\times 1}\\
    \vdots\\
    0_{n\times 1}
\end{pmatrix},
\begin{pmatrix}
    0_{n\times 1}\\
    J_{2\times 1}\\
    0_{n-2\times 1}\\
    0_{n\times 1}\\
    \vdots\\
    0_{n\times 1}
\end{pmatrix},
\begin{pmatrix}
    0_{n\times 1}\\
    0_{2\times 1}\\
    J_{n-2\times 1}\\
     0_{n\times 1}\\
    \vdots\\
    0_{n\times 1}
\end{pmatrix},
\ldots
\begin{pmatrix}
    0_{n\times 1}\\
    0_{2\times 1}\\
    0_{n-2\times 1}\\
     0_{n\times 1}\\
    \vdots\\
    J_{n\times 1}
\end{pmatrix}.$$
Since $\epsilon(K_{2n}\circ_{n} K_{2n}\circ_{n} \cdots\circ_{n} K_{2n}\backslash\{e\})$ is a  symmetric matrix of order $(n+1)l$, $\mathbb{R}^{(n+1)l}$ has an orthogonal basis consisting of eigenvectors of $\epsilon(K_{2n}\circ_{n} K_{2n}\circ_{n} \cdots\circ_{n} K_{2n}\backslash\{e\}).$
                
Hence, the remaining $l+2$ eigenvectors are spanned by these $l+2$ vectors.
Thus, if $\sigma_{3}$ is an eigenvalue of $\epsilon(K_{2n}\circ_{n} K_{2n}\circ_{n} \cdots\circ_{n} K_{2n}\backslash\{e\})$ with an eigenvector $\zeta_{3},$ of the form $\begin{pmatrix}
                p_{1}J_{n\times 1}\\
                p_{2}J_{2\times 1}\\
                p_{3}J_{n-2\times 1}\\
                p_{4}J_{n\times 1}\\
                \vdots\\
                p_{l+2}J_{n\times 1}
                
\end{pmatrix},$
where $p_{i}\in \mathbb{R},$ then from $\epsilon(K_{2n}\circ_{n} K_{2n}\circ_{n} \cdots\circ_{n} K_{2n}\backslash\{e\})\zeta_{3}=\sigma_{3}\zeta_{3},$ it follows that the remaining $l+2$ eigenvalues are the eigenvalues of the matrix,
$$B_{6}=\begin{pmatrix}
    n-1&2&n-2&n&n&n&\cdots &n\\
    n&2&0&2n&2n&2n&\cdots&2n\\
    n&0&0&2n&2n&2n&\cdots&2n\\
    n&4&2(n-2)&0&2n&2n&\cdots&2n\\
    n&4&2(n-2)&2n&0&2n&\cdots&2n\\
    n&4&2(n-2)&2n&2n&0&\cdots &2n\\
    \vdots&\vdots&\vdots&\vdots&\vdots&\vdots&\ddots&\vdots\\
    n&4&2(n-2)&2n&2n&2n&\cdots&0
\end{pmatrix}_{l+2\times l+2}.$$
Now, consider the vector $$v^{(i)}=(v_{1},v_{2},v_{3},v_{4},\ldots,v_{l+2})$$
such that $v_{4}=1,v_{i}=-1,5\leq i\leq l+2,$ and $v_{j}=0,$ for $v_{j}\neq v_{4},v_{i}.$ Then, all the vectors $v^{(i)}$ are linearly  independent and $B_{6}v^{(i)}=(-2n)x^{(i)}.$ Thus, $-2n$ is an eigenvalue of $B_{6}$ with multiplicity at least $l-2.$
Consider the equitable quotient matrix of $B_{6}$ is given by 
$$\tilde{B_{6}}=\begin{pmatrix}
    n-1&2&n-2&n(l-1)\\
    n&2&0&2n(l-1)\\
    n&0&0&2n(l-1)\\
    n&4&2(n-2)&2n(l-2)
\end{pmatrix}.$$
The characteristic polynomial of 
$\tilde{B_{6}}$ is, \begin{align*}
P_{\tilde{B_{6}}}(x)=&x^{4}-x^{3}(2nl-3n+1)-x^{2}(2n-2nl+3ln^{2}+2)+x(4n-12ln+2ln^{2}+2ln^{3}+4n^{2}-4n^{3})\\&+16n-24n^{2}+8n^{3}-16ln+16ln^{2}-4ln^{3}.
\end{align*}
We have, $$P_{\tilde{B_{6}}}(-2n)=8n(l-1)(5n-2)\neq 0.$$
Thus, $-2n$ is not a  root of $P_{\tilde{B_{6}}}(x).$ Using Theorem \ref{equitable quotient matrix eigenvalue thm}, the remaining eigenvalues of $B_{6}$ are  the eigenvalues of $\tilde{B_{6}}.$ Hence, the result follows.
\end{proof}

\begin{theorem}\label{case3 red edge deleting energy change , l=2}
Let $G=K_{2n}\circ_{n}K_{2n}\backslash \{e\},n\geq 3.$
    Then, eccentricity energy of $K_{2n}\circ_{n}K_{2n}\backslash \{e\}$ is greater than the eccentricity energy of $K_{2n}\circ_{n}K_{2n}.$
\end{theorem}
\begin{proof}
  By the suitable labelling of vertices, we get
    $$\epsilon(K_{2n}\circ_{n} K_{2n}\backslash\{e\})=\begin{pmatrix}
        (J-I)_{n\times n} &J_{n\times 2} &J_{n\times 2}&J_{n\times n}\\
        J_{2\times n} & 2(J-I)_{2\times 2} & 0_{2\times n-2}&2J_{2\times n}\\
        J_{n-2\times n}&0_{n-2\times 2}&0_{n-2\times n}-2&2J_{n-2\times n}\\
        J_{n\times n}&_2J{n\times 2}&2J_{n\times n-2}&0_{n\times n}
    \end{pmatrix}.$$
We see that,  $\epsilon(K_{2n}\circ_{n} K_{2n})\leq \epsilon(K_{2n}\circ_{n} K_{2n}\backslash\{e\}) .$
Moreover,  by Lemma \ref{e(G-e)is irreducible for case 3 edge, l=2},
$\epsilon(K_{2n}\circ_{n} K_{2n}\backslash\{e\})$ is irreducible. By Theorem \ref{Perron-Frobenius Theorem}, we have,
$$\rho_{\epsilon}( K_{2n}\circ_{n} K_{2n})\leq \rho_{\epsilon}(K_{2n}\circ_{n} K_{2n}\backslash\{e\}). $$
By Corollary \ref{inertia for n coalecence of K2n, for all cases}, we have 
$\epsilon( K_{2n}\circ_{n} K_{2n})$ has exactly one positive eigenvalue. 
Also, \begin{align*}
    E_{\epsilon}(K_{2n}\circ_{n} K_{2n}\backslash\{e\})=&\sum_{i=1}^{n}|\epsilon_{i}|,\\
    =&\sum_{\epsilon_{i}>0}\epsilon_{i}+\sum_{\epsilon_{i}\leq 0}\epsilon_{i},\\
    =&2\sum_{\epsilon_{i}>0}\epsilon_{i},\\
    \geq& 2\rho_{\epsilon}(K_{2n}\circ_{n} K_{2n}\backslash\{e\}).
\end{align*}

Hence, $$E_{\epsilon}( K_{2n}\circ_{n} K_{2n})=2\rho_{\epsilon}( K_{2n}\circ_{n} K_{2n})\leq 2\rho_{\epsilon}(K_{2n}\circ_{n} K_{2n}\backslash\{e\})\leq E_{\epsilon}(K_{2n}\circ_{n} K_{2n}\backslash\{e\}).$$
\end{proof}

\begin{theorem}  \label{energy change general case 2.red edge.}
    Let $G=K_{2n}\circ_{n} K_{2n}\circ_{n}\cdots \circ_{n} K_{2n},$ ($l$ copies of $K_{2n}$) $n \geq 4,$ $l \geq 4.$ Then,
$E_{\epsilon}(G)\leq E_{\epsilon}(G\backslash\{e\}).$   
\end{theorem}
\begin{proof}
    Consider the polynomial  \begin{align*}
P_{\tilde{B_{6}}}(x)=&x^{4}-x^{3}(2nl-3n+1)-x^{2}(2n-2nl+3ln^{2}+2)+x(4n-12ln+2ln^{2}+2ln^{3}+4n^{2}-4n^{3})\\&+16n-24n^{2}+8n^{3}-16ln+16ln^{2}-4ln^{3}.
\end{align*}
Let $\alpha_{1},\alpha_{2},\alpha_{3}$ and $\alpha_{4}$ be the roots of $P_{\tilde{B_{6}}}(x).$
We have the following relation between the roots and coefficients of $P_{\tilde{B_{6}}}(x):$
\begin{align}\label{relation between roots and coefficients in general l, case3}
    \alpha_{1}+\alpha_{2}+\alpha_{3}+\alpha_{4}=&2ln-3n+1,\\
\alpha_{1}\alpha_{2}\alpha_{3}\alpha_{4}=&16n-24n^{2}+8n^{3}-16ln+16ln^{2}-4ln^{3}\notag\\
=&4n(n-2)\big(-l(n-2)+2(n-1)\big).\notag
\end{align}
Since $n\geq 4, l\geq 4, \alpha_{1}\alpha_{2}\alpha_{3}\alpha_{4}<0.$
Therefore, $P_{\tilde{B_{6}}}(x)$ has either one or three negative roots.
We  see that $P_{\tilde{B_{6}}}(-x)$  has  exactly one sign  change. Thus, $P_{\tilde{B_{6}}}(x)$ can have at most one negative root.
Therefore, $P_{\tilde{B_{6}}}(x)$ has exactly one negative root.\\ 
Assume that $\alpha_{4}<0<\alpha_{3}\leq \alpha_{2}\leq \alpha_{1}.$\\

Now, \begin{align*}
    P_{\tilde{B_{6}}}(-2n)=&8n(l-1)(5n-2)>0,\\
    P_{\tilde{B_{6}}}(-2n+1)=&n(27-34n)-2+n^{2}l(43-10n)-28nl<0.
\end{align*}
Therefore, 
$-2n<\alpha_{4}<-2n+1.$ 
\begin{align*}
    2n-2<&-\alpha_{4},\\
\end{align*}
 \begin{align*}
    4n-4&<-2\alpha_{4},\\
    2nl+2nl-2n-n-1&<2nl+2nl-7n+3-2\alpha_{4},\\
                  &=2+2n(l-2)+2nl-3n+1-\alpha_{4}-\alpha_{4}.
\end{align*}
Using (\ref{relation between roots and coefficients in general l, case3}), \begin{align*}
    2n(l-1)+n(2l-1)-1&<2+2n(l-2)+\alpha_{1}+\alpha_{2}+\alpha_{3}-\alpha_{4}.\\
    \end{align*}
 Thus,
 \begin{align*}
     (n-1)+2n(l-1)+n(2l-1)-1&<(n-1)+2+2n(l-2)+\alpha_{1}+\alpha_{2}+\alpha_{3}-\alpha_{4}.
 \end{align*}
 By Corollary \ref{energy k coalencence for all cases} and Lemma \ref{e spectrum of coalecence , general case, l>4, case 3}, 
 $$E_{\epsilon}(G)\leq E_{\epsilon}(G\backslash\{e\}).$$
\end{proof}
\begin{remark}\label{remark general case of omited case in case 3 red edge}
    Using the same method as in the preceding theorem, it follows that for $l=3,n\geq 5,$
 $$E_{\epsilon}(K_{2n}\circ_{n}K_{2n}\circ_{n}K_{2n})=10n-2<E_{\epsilon}(K_{2n}\circ_{n}K_{2n}\circ_{n}K_{2n}\backslash\{e\}).$$
 Similerly, for $l>4,n=3,$  
 $$E_{\epsilon}(G)=12l-8<E_{\epsilon}(G\backslash\{e\}),$$ 
 where $G=K_{6}\circ_{3}K_{6}\circ_{3}\ldots \circ_{3}K_{6},$ $l>4$ copies of $K_{6}.$
\end{remark}

\begin{note}\label{note general case of omited case in case 3 red edge}
    For $l=3,n=3,$
    $$E_{\epsilon}(K_{6}\circ_{3}K_{6}\circ_{3}K_{6})=28.422<30.8582=E_{\epsilon}(K_{6}\circ_{3}K_{6}\circ_{3}K_{6}\backslash\{e\}).$$
    For $l=3,n=4,$
    $$E_{\epsilon}(K_{8}\circ_{4}K_{8}\circ_{4}K_{8})=38<40.8431=E_{\epsilon}(K_{8}\circ_{4}K_{8}\circ_{4}K_{8}\backslash\{e\}).$$
    For $l=4,n=3,$
    $$E_{\epsilon}(K_{6}\circ_{3}K_{6}\circ_{3}K_{6}\circ_{3}K_{6})=40<42.2584=E_{\epsilon}(K_{6}\circ_{3}K_{6}\circ_{3}K_{6}\circ_{3}K_{6}\backslash\{e\}).$$
\end{note}


From Theorems \ref{case3 red edge deleting energy change , l=2},  \ref{energy change general case 2.red edge.}, Remark \ref{remark general case of omited case in case 3 red edge}, and Note \ref{note general case of omited case in case 3 red edge},
we get the following result.
\begin{theorem}\label{final conclusion for case 3 red edge}
    Let $G=K_{2n}\circ_{n} K_{2n}\circ_{n}\cdots \circ_{n} K_{2n},$ (l copies of $K_{2n}$) $n\geq 3,$ $l\geq 2.$ Then,
$E_{\epsilon}(G)\leq E_{\epsilon}(G\backslash\{e\}).$  
\end{theorem}

Combining Theorems \ref{final conclusion for case 1 green edge}, \ref{final conclusion for case 2 blue edge}, and \ref{final conclusion for case 3 red edge}, we obtain the main result of our study.
\begin{theorem}
    Let $G=K_{2n}\circ_{n} K_{2n}\circ_{n}\cdots \circ_{n} K_{2n},$ (l copies of $K_{2n}$) $n\geq 3,$ $l\geq 2,$ and $e$ be any edge in $G.$ Then,
$E_{\epsilon}(G)\leq E_{\epsilon}(G\backslash\{e\}).$ 
\end{theorem}
\section*{Acknowledgement}
The research of Anjitha Ashokan is supported by the University Grants Commission of India under the beneficiary code BININ05086971.
\section*{Declarations}
 On behalf of all authors, the corresponding author states that there is no conflict of interest.

 \bibliographystyle{plain}
  \bibliography{reference}
\end{document}